\documentclass[11pt,twoside]{article}
\usepackage[utf8]{inputenc}
\usepackage[margin=1in]{geometry}
\usepackage[T1]{fontenc}
\usepackage{amsmath,amsthm,amssymb}
\usepackage{newcent,microtype}
\usepackage[hidelinks]{hyperref}
\usepackage{mathtools}

\swapnumbers
\newtheorem{theorem}{Theorem}[section]
\newtheorem{corollary}[theorem]{Corollary}
\newtheorem{definition}[theorem]{Definition}
\newtheorem{lemma}[theorem]{Lemma}
\newtheorem{observation}[theorem]{Observation}

\newcommand\ie{{\em i.e., }}

\DeclareMathOperator\ad{ad}

\DeclareMathOperator\Aut{Aut}
\DeclareMathOperator\C{C}

\DeclareMathOperator\End{End}

\DeclareMathOperator\id{id}
\DeclareMathOperator\im{im}

\DeclareMathOperator\Sym{Sym}

\newcommand\la{\langle}

\newcommand\ra{\rangle}

\newcommand\size[1]{\lvert #1\rvert}

\newcommand\CC{\mathbb{C}}

\newcommand\kk{\mathbf{k}}

\newcommand\QQ{\mathbb{Q}}
\newcommand\RR{\mathbb{R}}
\newcommand\ZZ{\mathbb{Z}}

\newcommand\op{\mathrm}

\renewcommand\frak{\mathfrak}

\newcommand\al{\alpha}

\newcommand\dl{\delta}		
\newcommand\ep{\varepsilon}

\newcommand\kp{\kappa}
\newcommand\lm{\lambda}

		\newcommand\Om{\Omega}

\allowdisplaybreaks
\newcommand\nhood{\C^\op{c}}
\newcommand\commute{\C^\times}
\newcommand\rt{\mathrm}

\begin{document}
\abovedisplayskip=0.5em
\belowdisplayskip=0.5em

\title{Fusion rules from root systems I: case $\rt A_n$}
\author{F. Rehren}
\date{\today}
\maketitle

\begin{abstract}
	Axial algebras are commutative algebras generated by idempotents;
	they generalise associative algebras
	by allowing the idempotents to have additional eigenvectors,
	controlled by fusion rules.
	If the fusion rules are $\ZZ/2$-graded,
	axial algebras afford representations of transposition groups.
	We consider axial representations of Weyl groups of simply-laced root systems,
	which are examples of regular $3$-transposition groups.
	We introduce coset axes,
	a special class of idempotents based on embeddings of transposition groups,
	and use them to study the propagation of fusion rules in axial algebras,
	for root system $\rt A_n$.
	This is related to the construction of lattice vertex operator algebras
	and we show it reflects on the fusion of modules for the Virasoro algebra
	when we specialise our construction.
\end{abstract}

\renewcommand{\arraystretch}{1.75}
\section{Introduction}
\label{sec intro}
	Inside a simple construction
	there is room for greater beauty and complexity to grow.
	This note studies the fusion rules that arise
	inside certain nonassociative `axial' algebras
	stemming from simply-laced root systems.
	Our construction is an axial representations of Weyl groups,
	or more generally transposition groups.
	In special cases, we find the emerging complexity in our algebras
	is in intimate relation to the fusion rules for modules of the Virasoro algebra,
	an infinite-dimensional Lie algebra.

	Our theory starts with root systems.
	In the 80s, these led to the first constructions of vertex operator algebras (VOAs) \cite{flm},
	which are the inspiration behind axial algebras.
	Root systems are a powerful tool for analysing large families of VOAs,
	via, for example, so-called frames.
	These are related to the surprising fact that the Griess algebra
	(the celebrated 196884-dimensional algebra
	whose automorphism group was the first realisation of the Monster group \cite{griess})
	contains an associative subalgebra of dimension $48$ \cite{meyer-neutsch},
	and this is the largest possible associative subalgebra \cite{miyamoto}.
	Further developments in this direction were made in \cite{dlmn},
	\cite{dmz} and \cite{dgh}.

	In particular \cite{dlmn} discusses the weight-$2$ subalgebra
	of so-called lattice VOAs (for lattices of simply-laced root systems)
	and shows that they yield interesting representations of the Virasoro algebra.
	This is one point of departure for us.
	Another is \cite{hrs}, which began the analysis
	of axial algebras related to $3$-transposition groups.
	In combination of the two,
	for any $3$-transposition group 
	we construct a family of representations, on axial algebras,
	parametrised by $\al$ an eigenvalue.
	We find eigenvectors and their fusion rules for some special idempotents
	in the case of the Weyl group of $\rt A_n$.
	This is the first time, as far as we are aware,
	that a description of families of fusion rules occurring in an algebra
	has been given.
	Then, in specialising $\al$,
	we recover additional details about the \cite{dlmn} construction
	and glimpse a new point of view on the so-called minimal models of the Virasoro algebra.
	
	We now outline this note and its results.
	The summary is followed by a discussion.

	Chapter \ref{sec bg}, which can be skimmed by the expert for notation,
	establishes the necessary notions and facts as follows.
	Section~\ref{ss axial} first defines fusion rules,
	which describe the multiplication of submodules of an algebra
	analogous to the description of multiplication of elements
	by structure constants.
	An axis is an idempotent the multiplication of whose eigenspaces is described by a specified fusion rules;
	an axial algebra is a nonassociative algebra generated by axes.
	If the fusion rules satisfy a $\ZZ/2$-grading property,
	we can also study the automorphism groups of axial algebras
	and we define axial representations of transposition groups.
	We also briefly describe axial representations of $3$-transposition groups;
	this is related to \cite{hrs}.
	Such representations are an essential constituent of this paper.
	Section~\ref{ss cox weyl}	summarises results about transposition-, Coxeter- and Weyl groups.
	Section~\ref{ss vir} sketches facts about the Virasoro algebra;
	more comprehensive treatment could be found in \cite{cft}.

	Chapter \ref{sec coset axes} contains the substance of this note.
	We introduce coset axes in Definition \ref{def coset axis},
	which are idempotents in the representation of a transposition group
	derived from embeddings of subtransposition groups.
	Together with some straightforward results on identities of algebras,
	this is Section~\ref{ss id ca}.
	We then specialise to the case of the root system $\rt A_n$ for Section~\ref{ss eigval} and the remainder of this paper.
	In Theorem \ref{thm fus rules} we give the $5$ possible eigenvalues and their fusion rules
	for coset axes arising from embeddings of root systems $\rt A_\ell\subseteq\rt A_m$,
	which we can deduce from the properties of identities of subalgebras.
	In Section~\ref{ss cc},
	we introduce a bilinear form
	for some simple results on the length of coset axes.
	As mentioned in \cite{dlmn},
	this is evocative of the celebrated coset construction of \cite{gko},
	but our construction depends on the rank of the root system
	with fixed level equal to $1$ for its affine lie algebra.
	In Section~\ref{ss apps},
	we discuss the specialisation of our construction
	that allows an embedding in the weight-$2$ subalgebra of lattice VOAs,
	and recovers the situation of \cite{dlmn}.
	Here our results on eigenvalues and fusion rules
	become, in Observation \ref{obs hw identification},
	results about the fusion rules of Virasoro algebras,
	which are Clebsch--Gordan-type coefficients
	for the decomposition into irreducibles
	of the fusion of two modules.
	In particular, we recover the highest weights
	for some points in the bottom left corner
	of the Kac table
	(which classifies the possible highest weights, see, e.g.~\cite{cft} \S7.3.3)
	for a unitary irreducible representation of central charge less than $1$.
	Finally, in Section~\ref{ss double fischer},
	we examine the additional eigenvalues in the $\rt A_n$ case (Lemma \ref{lem id+ extra eigv})
	found in a construction generalising weight-$2$ subalgebras of lattice VOAs.
	Its specialisation picks out two further points,
	in Observation \ref{obs hat hw identification}, of the Kac table.

	The small, fixed window on the Kac table that is opened
	is evidently in search of explanation.
	Speculatively, perhaps the eigenvalues of the identities $i,j$
	for the coset axis $i - j$
	line the sides of the Kac table.
	Observe that we only observe points $h^c_{r,s}$ whose coordinates $(r,s)$
	have $r,s$ both odd.
	A consequence of the fusion rules for the Virasoro algebra
	is that there is no $\ZZ/2$-grading, in this case,
	unless one of $r,s$ is even;
	this follows from \cite{wang}.
	We do not know of a construction $\tilde A_\Phi(G,D)$
	that generalises $A^\al_3(G,D)$ and $\hat A^\al_3(G,D)$
	and, in its specialisation,
	recovers any further highest weights in the Kac table
	for coset axes when $(G,D) = W(\rt A_n)$.

	The equation \eqref{eq cc 1/32} seems to be a hint
	that a more general construction of lattice VOAs is possible.
	In particular, the coset axis $x_{\rt A_2/\rt A_1}$ in $\smash{A^{1/32}_3(\rt A_n)}$
	has central charge $\frac{21}{22}$, matching $\op{Vir}_{c(12,11)}$,
	and the observed eigenvalues of the former
	match the highest weights of the latter.
	The coset axes $x_{\rt A_i/\rt A_{i-1}}$, $i\geq3$,
	have central charge greater than $1$,
	so if such a lattice VOA exists
	we also have observations on highest weights in the continuous sectors.

	We have focused on results for the root system $\rt A_n$.
	The case $\rt D_n$ has a different flavour,
	since the coset axes have central charge $1$ in the specialisation
	(see \eqref{eq f quarter Dm}),
	and their fusion rules are $\ZZ/2$-graded,
	so these axes give rise to Miyamoto automorphisms of the algebra.
	This case will be considered in a later work.

	I thank
	Alonso Castillo-Ramirez
	and Oliver Gray
	for many stimulating discussions.

	\clearpage

\section{Background} \label{sec bg}
	\subsection{Fusion rules and axial things} \label{ss axial}
		A fusion rules is a simple bookkeeping device
		that we use to describe properties of axes.
		\begin{definition}
			\label{def fus rules}
			A {\em fusion rules} over a field $\kk$
			is a finite collection of {\em eigenvalues} $\Phi\subseteq\kk$
			and a map $\star\colon\Phi\times\Phi\to2^\Phi$.
			We will often allow $\Phi$ to stand for the pair $(\Phi,\star)$.
		\end{definition}

		We give three examples of fusion rules:
		$\Phi_{\rm ass} = (\{1,0\},\star_{\rm ass})$,
		$\Phi^\al_3 = (\{1,0,\al\},\star_3)$ for any $\al\in\kk\smallsetminus\{1,0\}$,
		and $\frak V(5,3) = (\{1,0,\frac{1}{10},-\frac{1}{40},\frac{3}{8}\},\star_{(5,3)})$.
		The fusion rules are below.
		Since these fusion rules, as all the fusion rules that we will consider,
		have a commutative map $\star$,
		we omit the symmetric entries.
		\begin{center}
		\begin{tabular}{|c||c|c|}
			\hline
				$\star_{\rm ass}$	& $1$ & $0$ \\
			\hline\hline
				$1$ & $\{1\}$ & $\emptyset$ \\
			\hline
				$0$ &  & $\{0\}$ \\
			\hline
		\end{tabular}
		\begin{tabular}{|c||c|c|c|}
			\hline
				$\star_3$	& $1$ & $0$ & $\al$ \\
			\hline\hline
				$1$ & $\{1\}$ & $\emptyset$ & $\{\al\}$ \\
			\hline
				$0$ &  & $\{0\}$ & $\{\al\}$ \\
			\hline
				$\al$ &  &  & $\{1,0\}$ \\
			\hline
		\end{tabular}
		\begin{tabular}{|c||c|c|c|c|c|}
			\hline
				$\star_{(5,3)}$ & $1$ & $0$ & $\frac{1}{10}$ & $-\frac{1}{40}$ & $\frac{3}{8}$ \\
			\hline\hline
				$1$ & $\{1\}$ & $\emptyset$ & $\{\frac{1}{10}\}$ & $\{-\frac{1}{40}\}$ & $\{\frac{3}{8}\}$ \\
			\hline
				$0$ &  & $\{0\}$ & $\{\frac{1}{10}\}$ & $\{-\frac{1}{40}\}$ & $\{\frac{3}{8}\}$ \\
			\hline
				$\frac{1}{10}$ &  &  & $\{1, 0, \frac{1}{10}\}$ & $\{-\frac{1}{40},\frac{3}{8}\}$ & $\{-\frac{1}{40}\}$ \\
			\hline
				$-\frac{1}{40}$ &  &  &  & $\{1, 0, \frac{1}{10}\}$ & $\{\frac{1}{10}\}$ \\
			\hline
				$\frac{3}{8}$	& 	& 	&  &  & $\{1, 0\}$ \\
			\hline
		\end{tabular}
		\end{center}

		The next definition will clarify how we use fusion rules.
		Let $A$ be a nonassociative commutative algebra over $\kk$,
		that is, a $\kk$-vector space
		with a not-necessarily-associative commutative multiplication.
		Denote by $\ad(a)\in\End(A)$ the map $b\mapsto ab$
		of left-multiplication (which coincides with right-multiplication)
		for $a\in A$.
		We use the notation $A^a_\lm = \{b\in A\mid \ad(a)b = \lm b\}$
		for the set of $\lm$-eigenvectors of $\ad(a)$ in $A$.
		The element $a\in A$ is said to be {\em semisimple}
		if $A$ decomposes as a direct sum of eigenspaces with respect to $\ad(a)$,
		that is, if there exist $\lm_1,\dotsc,\lm_n\in\kk$ such that
		\begin{equation}
			A = A^a_{\lm_1}\oplus A^a_{\lm_2}\oplus\dotsm\oplus A^a_{\lm_n}.
		\end{equation}
		Furthermore, we fix the following notations:
		\begin{equation}
			A^a_{\mu_1,\dotsc,\mu_m} = A^a_{\{\mu_1,\dotsc,\mu_m\}}
			 = A^a_{\mu_1} + A^a_{\mu_2} + \dotsm + A^a_{\mu_m}.
		\end{equation}
		Finally, recall that $a$ is an {\em idempotent} if $aa = a$.
		
		\begin{definition}
			\label{def axis}
			An element $x\in A$ is a {\em $\Phi$-axis},
			for $\Phi$ a fusion rules,
			if $x$ is a semisimple idempotent for which
			$A = A^x_\Phi = \bigoplus_{\lm\in\Phi} A^x_\lm$
			such that
			$A^x_\lm A^x_\mu \subseteq A^x_{\lm\star\mu}$.
		\end{definition}

		\begin{definition}
			\label{def axial alg}
			A nonassociative commutative algebra $A$
			is a {\em $\Phi$-axial algebra}
			if $A$ can be generated by a set of $\Phi$-axes in $A$.
		\end{definition}

		Let us give a simple example.
		Suppose that $A$ is an {\em associative} commutative algebra,
		containing a semisimple idempotent $x$.
		Then simple calculations show that $A = A^x_1 \oplus A^x_0$,
		and the product of eigenvectors matches $\star_{\rm ass}$,
		whence $x$ is a $\Phi_{\rm ass}$-axis.
		We delay slightly before giving a less shallow example:
		of a $\Phi^\al_3$-axial algebra.

		If $a\in A$ is a $\Phi$-axis,
		the existence of a grading on $\Phi$
		will lead to the existence of automorphisms of $A$.
		This is of crucial interest to us.
		\begin{definition}
			\label{def super rules}
			A fusion rules $(\Phi,\star)$ is {\em $\ZZ/2$-graded}
			if there exists a partition $\Phi_+\cup\Phi_-$ of the eigenvalues $\Phi$
			such that, for $\ep,\ep'\in\{+,-\}$ and $\lm\in\Phi_\ep,\mu\in\Phi_{\ep'}$,
			$\lm\star\mu\subseteq\Phi_{\ep\ep'}$.
		\end{definition}
		\begin{definition}
			\label{def miyamoto}
			If $\Phi$ is a $\ZZ/2$-graded fusion rules,
			realised as $\Phi = \Phi_+ \cup \Phi_-$,
			and $x\in A$ is a $\Phi$-axis,
			then the {\em Miyamoto involution $\tau(x)$} for $x$
			is an automorphism of $A$ acting as
			$$ a^{\tau(x)} = \begin{cases}
				a & \text{ if } a\in A^x_{\Phi_+} \\
				-a & \text{ if } a\in A^x_{\Phi_-} \end{cases}, $$
			and extending linearly from the eigenvectors to the entire space.
		\end{definition}

		Suppose that $\Phi$ is $\ZZ/2$-graded,
		and that $A$ is a $\Phi$-axial algebra.
		Then the subgroup of automorphisms generated by Miyamoto involutions
		has an interesting group-theoretic structure.

		\begin{definition}
			\label{def trgp}
			A {\em transposition group} is a pair $(G,D)$,
			where $G$ is a group and $D\subseteq G$ is a normal generating subset of involutions from $G$.
		\end{definition}

		Thus, if $D$ is the (normal closure) of a set of Miyamoto involutions
		from a set of $\Phi$-axes in the algebra $A$,
		then $(\la D\ra,D)$ forms a transposition group,
		lying inside $\Aut(A)$.

		The best-known example of transposition groups
		are Fischer's $3$-transposition group
		namely those in which the order of the product of any two elements of $D$ is at most $3$.
		They have been classified \cite{cuypershall}.
		For a general transposition group and $\ZZ/2$-graded fusion rules,
		we have 

		\begin{definition}
			\label{def axial rep}
			A $\Phi$-{\em axial representation}
			of a transposition group $(G,D)$
			is a mapping $\rho\colon D\to A$,
			$A$ a $\Phi$-axial algebra and $\Phi$ $\ZZ/2$-graded,
			such that for all $d\in D$, $d^\rho\in A$ is a $\Phi$-axis,
			the algebra $A$ is generated by the image $\{d^\rho\mid d\in D\}$ of $\rho$,
			and for $c,d\in D$, $(c^\rho)^{\tau(d^\rho)} = (c^d)^\rho$.
		\end{definition}

		The $3$-transposition groups in particular
		are closely related to the $\Phi^\al_3$ fusion rules,
		which can be $\ZZ/2$-graded as $(\Phi^\al_3)_+=\{1,0\}$ and $(\Phi^\al_3)_-=\{\al\}$.
		It turns out that Miyamoto automorphisms from any $\Phi^\al_3$-axes
		form a $3$-transposition group \cite{hrs},
		and any $3$-transposition group has a faithful $\Phi_3$-axial representation.
		We describe this in more detail; it will be used in the sequel.

		\begin{definition}
			\label{def fischer rep}
			For $(G,D)$ a $3$-transposition group,
			define $A = A^\al_3(G,D)$ as the algebra
			with underlying vector space $\QQ\{ d^\rho \mid d\in D\}$
			and multiplication, for some $\al\in\kk$,
			$$  c^\rho  d^\rho  = \begin{cases}
				 d^\rho  & \text{ if } c = d \\
				0 & \text{ if } [c,d] = 1 \\
				\frac{\al}{2}( c^\rho  +  d^\rho  - (c^d)^\rho) & \text{ otherwise.}
			\end{cases} $$
			Denote the mapping $d\mapsto d^\rho $ by $\rho = \rho^\al_3\colon D\to A$.
		\end{definition}

		That $A$ is commutative follows from the fact that
		if $c\neq d$ and $[c,d]\neq 1$, then $\size{cd} =3$,
		$\la c,d\ra\cong{\rm Sym}(3)$ and $c^d = d^c$.
		Then $d^\rho$ is a $\Phi^\al_3$-axis for all $d$,
		and the representation is constructed to satisfy
		$(c^d)^\rho = (c^\rho)^{\tau(d^\rho)}$,
		so $\rho$ satisfies the Definition \ref{def axial rep}
		of a $\Phi^\al_3$-axial representation.
		For further discussion, see also \cite{hrs}.
		The algebra $A$ is called the Matsuo algebra
		associated to $(G,D)$.

		There are two interesting properties of $A$ that we will remark upon.
		Firstly, $A$ admits a bilinear form $\la,\ra$, defined by
		\begin{equation}
			\label{eq def form}
			\la c^\rho , d^\rho \ra = \begin{cases}
				1 & \text{ if }c = d \\
				0 & \text{ if }[c,d]=1 \\
				\frac{\al}{2} & \text{ otherwise.}
			\end{cases}
		\end{equation}
		(The product $c^\rho d^\rho$ might now be described more concisely as
		$\la c^\rho , d^\rho \ra( c^\rho + d^\rho -(c^d)^\rho)$, in all cases.)
		
		\clearpage

		Secondly, the $\Phi^\al_3$-axes $ d^\rho $, $d\in D$, in $A$
		are {\em primitive} idempotents,
		that is,
		the $1$-eigenspace of $ d^\rho $ is precisely $\la d^\rho \ra$,
		hence $1$-dimensional.
		In particular a primitive idempotent
		is not decomposable as a sum of pairwise annihilating idempotents.
		(In an associative algebra, a nonprimitive idempotent
		whose $1$-eigenspace has dimension $d$
		has a decomposition into $d$ pairwise annihilating idempotents,
		but this is not necessarily the case in our situation.)
		We have a greater interest in primitive idempotents,
		since properties of nonprimitive idempotents may be easily deduced
		from those of the idempotents in their decomposition.
		
		Finally a general fact of interest:
		\begin{lemma}
			\label{lem f3 axial rep}
			Let $\phi\colon D\to A$ be a $\Phi^\al_3$-axial representation
			of a $3$-transposition group $(G,D)$.
			Then $A$ is a quotient of $A^\al_3(G,D)$.
			\qed
		\end{lemma}

	\subsection{Coxeter, Weyl and transposition groups} \label{ss cox weyl}
		A {\em Coxeter group} is a group $G$ satisfying a certain presentation
		specified by a matrix $(C_{ij})_{1\leq i,j\leq n}$:
		\begin{equation}
			G = \la s_1,\dotsc,s_n\mid (s_is_j)^{C_{ij}} = 1 \ra,
			\text{ s.t. }
			C_{ii} = 1, C_{ij} = C_{ji} \quad \forall i,j.
		\end{equation}
		Thus $G$ is generated by involutions,
		and therefore with $D = {s_1}^G\cup {s_2}^G\cup \dotsm \cup {s_n}^G$,
		the Coxeter group $G$ can be realised
		as a transposition group $(G,D)$ in a standard way.

		We will concentrate on the {\em Weyl groups}:
		the Coxeter groups satisfying the crystallographic condition,
		thus being associated to root systems and simple Lie algebras.
		(Suppose that $(C_{ij})_{1\leq i,j\leq n}$ is the Cartan matrix
		associated to a simple Lie algebra.
		The corresponding Coxeter matrix $C'$ has entries $C'_{ii} = 1$
		and, for $i\neq j$, $C'_{ij}=2, 3, 4, 6$
		if $C_{ij}C_{ji} = 0, 1, 2, 3$ respectively.)

		The Weyl group of a root system ${\rt X_n}$
		may be denoted $W(\rt X_n) = (G(\rt X_n),D(\rt X_n))$,
		the latter emphasising its structure as a transposition group.
		Since this is unambiguous, we will also use the notation
		$A^\al_3(\rt X_n)$ for the $A^\al_3(G(\rt X_n),D(\rt X_n))$ of Definition \ref{def fischer rep}.
		The Weyl group of the root system $\rt A_n$
		is $(\Sym(n+1),(1,2)^{\Sym(n+1)})$;
		of $\rt D_n$, $n\geq 4$, it is $G(\rt D_n) = 2^{n-1}:\Sym(n)$,
		with transpositions $D(\rt D_n)$ the preimage of $(1,2)^{\Sym(n)}$ in $G(\rt D_n)$
		under the quotient by the normal subgroup $2^{n-1}$.

		In particular, we are mostly interested in simply-laced root systems here:
		those of types $\rt A_n,\rt D_n,\rt E_n$.
		We list some results.
		The root system $\rt X_n$ is simply-laced
		if and only if $W(\rt X_n)$ is a $3$-transposition group.
		Recall that the {\em rank} of a root system ${\rt X_n}$ is $n$.
		The {\em Coxeter number} $h_{\rt X_n}$ is equal to the dual Coxeter number $h_{\rt X_n}^\vee$
		exactly when ${\rt X_n}$ is simply-laced.
		The Coxeter number of $\rt A_n,\rt D_n,\rt E_6,\rt E_7,\rt E_8$
		is $n+1,2n-2,12,18,30$ respectively.
				
		If $(G,D)$ is transposition group,
		then, with respect to any $d\in D$,
		$D$ has a partition 
		\begin{equation}
			\begin{aligned}
				\label{eq nhood 3trgp}
				D & = \{d\}\cup \nhood_D(d)\cup \commute_D(d), \text{ where }\\
				\nhood_D(d) & = \{c\in D\mid [c,d]\neq1\} = D - \C_D(d), \\
				\commute_D(d) & = \{c\in D\smallsetminus\{d\}\mid [c,d]=1\} = \C_D(d)\setminus\{d\}.
			\end{aligned}
		\end{equation}
		In particular, the set $\nhood_D(d)$
		is the collection of neighbours of $d$ in the noncommuting graph on $D$.
		It will be very helpful to us when the sizes of the sets in the partition
		are constant for changing $d\in D$.
		In this case, the commuting and noncommuting graphs of $D$ are regular graphs,
		and by extension we will then also call $(G,D)$ {\em regular}.
		In particular, $(G,D)$ is called {\em $k$-regular}
		if $\size{\nhood_D(d)}=k$ for all $d\in D$.
		It is clear that if $D$ is a single conjugacy class in $G$,
		then $\size{C_D(d)}=\size{C_{D^g}(d^g)}=\size{C_D(d^g)}$ for all $g\in G$
		and this implies that $(G,D)$ is regular.
		It is well-known (cf. Chap. IV, 1.11, Prop. 32 of \cite{bourb})
		that the Weyl group of a simply-laced root system ${\rt X_n}$ is regular:
		namely, $2h_{\rt X_n}-4$-regular.

	\subsection{The Virasoro algebra} \label{ss vir}
		The Virasoro algebra
		is a central extension of the Witt algebra
		with the following presentation:
		\begin{equation}
			\op{Vir} = \CC\{L_i\}_{i\in\ZZ}\oplus\CC c,\quad
			[L_i,L_j] = (i-j)L_{i+j} + \frac{1}{12}i(i^2-1)\dl_{i,-j}c,\quad
			[c,\op{Vir}] = 0.
		\end{equation}
		Since $c$ is central, it acts by a scalar on any module of $\op{Vir}$.
		This scalar is also denoted $c$.
		It turns out that the value of $c$ is decisive in determining the representations of $\op{Vir}$.
		Denote by $\op{Vir}_c$ the Virasoro algebra restricted in its possible actions on a given module to $c\in\op{Vir}$ taking the value $c\in\CC$ on the module.
		In particular, $\op{Vir}_c$ only has finitely many irreducible modules
		(up to isomorphism)
		exactly when $c$ is \cite{wang}
		\begin{equation}
			\label{eq cc(p,q) vir}
			c = c(p,q) = 1 - \frac{6(p-q)^2}{pq} \text{ for } p,q\in\{2,3,\dotsc\} \text{ coprime.}
		\end{equation}
		In this case, the modules may be distinguished
		by the possible eigenvalues of $L_0$ on their highest weight vectors,
		for which the only possibilities are
		\begin{equation}
			\label{eq hw vir}
			h^{c(p,q)}_{r,s} = \frac{(sp-rq)^2 - (p-q)^2}{4pq},\quad
			1\leq r< p, 1\leq s< q.
		\end{equation}
		The tensor product of two modules is again a module,
		but the central charges (the action of $c$) sum.
		Instead, {\em fusion} of modules allows a decomposition
		into irreducibles of the same, original, central charge.
		This can be described on the level of the highest weights:
		namely \cite{wang},
		\begin{equation}
			h^{c(p,q)}_{r,s}\star h^{c(p,q)}_{t,u} \subseteq
			\biggl\{ h^{c(p,q)}_{v,w} \bigg\vert
				\begin{aligned}
					& 1 + \size{r-t} \leq v\leq \min\{r+t-1,2p-r-t-1\},\quad v\equiv_2 1+r+t \\
					& 1 + \size{s-u} \leq w\leq \min\{s+u-1,2q-s-u-1\}, w\equiv_2 1+s+u
				\end{aligned}
			\biggr\}.
		\end{equation}
		The fusion rules $\frak V(5,3)$ that we gave
		are derived from those of $\op{Vir}_{c(5,3)}$.
		The recipe for so doing is that the eigenvalues of the new fusion rules
		are $\{\frac{h}{2}\mid h \text{ h.w. }\op{Vir}_c\}\cup\{1\}$,
		and the rules $\star$ are the same,
		with some additional rules for $1$.
		The rescaling comes from VOAs.

\section{Coset axes} \label{sec coset axes}
	\subsection{Identities and coset axes}\label{ss id ca}
		\begin{definition}
			\label{def coset axis}
			Suppose that $\rho\colon D\to A$ is a $\Phi$-axial representation
			of the transposition group $(G,D)$.
			Let $(K,F)\subseteq(H,E)\subseteq(G,D)$ be subtransposition groups,
			that is, $K\subseteq H\subseteq G$
			and $E = H\cap D, F = K\cap E$. 
			Then, if the subalgebras $\la E^\rho\ra$ and $\la F^\rho\ra$
			admit identity elements $\id_{E^\rho}\in\la E^\rho\ra$
			and $\id_{F^\rho}\in\la F^\rho\ra$ respectively,
			the {\em coset axis of $H/K$} is
			$$ e_{H/K} = \id_{E^\rho} - \id_{F^\rho}. $$
		\end{definition}

		Coset axes have a distinguished r\^ole in part
		because they afford decompositions of the identity idempotents:
		\begin{lemma}
			\label{lem ca decomp}
			$\id_{E^\rho}$ decomposes as
			$\id_{E^\rho} = e_{H/K} + \id_{F^\rho}$
			into pairwise annihilating idempotents.
		\end{lemma}
		\begin{proof}
			To show that $e_{H/K}$ and $\id_{F^\rho}$ annihilate:
			$ e_{H/K}\id_{F^\rho} = (\id_{E^\rho}-\id_{F^\rho})\id_{F^\rho}
				= \id_{F^\rho}-\id_{F^\rho} = 0 $.
			Also, $e_{H/K}$ is an idempotent, for
			\begin{equation*}
				e_{H/K}e_{H/K}
				= \id_{E^\rho}\id_{E^\rho} - 2\id_{E^\rho}\id_{F^\rho} + \id_{F^\rho}\id_{F^\rho}
				= \id_{E^\rho}-\id_{F^\rho} = e_{H/K}. \qedhere
			\end{equation*}
		\end{proof}
		Of course the terms $e_{H/K}$ and $\id_{F^\rho}$
		may themselves admit further decompositions.
		Recall from Section \ref{ss axial} that an idempotent is primitive
		if its $1$-eigenspace is $1$-dimensional,
		and that in this case it cannot be decomposed.

		\begin{lemma}
			\label{lem ca prim}
			The coset axis $e_{H/K}$ is primitive
			only if $\dim A^{\id_{E^\rho}}_1\cap A^{\id_{F^\rho}}_0 = 1$.
		\end{lemma}
		\begin{proof}
			Evidently, if $x\in A^{\id_{E^\rho}}_1\cap A^{\id_{F^\rho}}_0$,
			then $e_{H/K}x = 1x - 0x = x$, so $x\in A^{e_{H/K}}_1$.
		\end{proof}

		We say that the embedding $(K,F)\subseteq(H,E)$
		is {\em primitive (with respect to $\rho$)}
		if $e_{H/K}$ is primitive.

		Suppose that $e$ is an idempotent in the algebra $A$
		and admits a decomposition into pairwise annihilating idempotents $e_1,\dotsc,e_n$.
		Then the algebra $B\subseteq A$ spanned by the $e_1,\dotsc,e_n$
		is evidently associative,
		since the algebra $\la e_i\ra$ is trivially associative
		and $B$ is isomorphic to $\la e_1\ra\oplus\dotsm\oplus\la e_n\ra$.

		Hence coset axes are related to associative subalgebras.
		In fact, consider the inclusions $(1,\{\})\subseteq(C_2,\{t\})\subseteq(G,D)$,
		where we allow $(1,\{\})$ as a transposition group for the moment,
		and $t\in D$ generates $C_2$.
		Then we extend the definition of a coset axis to encompass
		$e_{C_2/1} = t^\rho$.
		(As $t^\rho$ is a primitive axis,
		the embedding $(1,\{\})\subseteq(C_2,\{t\})$ is considered primitive.)
		Thus a chain of embeddings
		$(1,\{\})\subset\dotsm\subset(H,E)\subset\dotsm\subset(G,D)$
		yields a collection of coset axes spanning an associative subalgebra of $A$.
		If it can be shown that each embedding of the chain
		yields a proper embedding of subalgebras
		with $\id_{E^\rho}\neq\id_{F^\rho}$ for any $(H,E)\subset(K,F)$,
		then the dimension of the associative subalgebra
		is the number of nontrivial transposition groups in the chain.
		The associative subalgebra is maximal
		if every embedding in the chain is primitive.

		Throughout the rest of this section,
		we specialise to the case where $A$ is isomorphic
		to the $\Phi^\al_3$-axial algebra $A^\al_3(G,D)$,
		with representation $\rho = \rho^\al_3$ of the $3$-transposition group $(G,D)$.

		\begin{lemma}
			\label{lem f3 alg unital}
			Set $\sum D^\rho = \sum_{d\in D}d^\rho$. Then
			$\displaystyle\smash{d^\rho\sum D^\rho = (1 + \frac{\al}{2}\size{\nhood_D(d)})d^\rho}$.
		\end{lemma}
		\begin{proof}
			From \eqref{eq nhood 3trgp}, we have that
			\begin{equation}
				d^\rho\sum D^\rho = d^\rho d^\rho
				+ d^\rho\sum_{\mathclap{c\in\nhood_D(d)}}c^\rho
				+ d^\rho\sum_{\mathclap{c\in \commute_D(d)}}c^\rho
				= d^\rho + \frac{\al}{2}\sum_{\mathclap{c\in\nhood_D(d)}}(d^\rho + c^\rho - (c^d)^\rho)
				+ \sum_{\mathclap{c\in \commute_D(d)}}0
				= d^\rho + \frac{\al}{2}\size{\nhood_D(d)}d^\rho,
			\end{equation}
			where the last equality follows since,
			as $c$ ranges over $\nhood_D(d)$,
			so does $c^d$,
			and hence in the sum over $\nhood_D(d)$ the contributions of terms $c^\rho$ and $(c^d)^\rho$ cancel.
		\end{proof}

		\begin{corollary}
			\label{cor f3 id}
			If $(G,D)$ is $k$-regular, then
			$A^\al_3(G,D)$ admits an identity
			$\displaystyle\smash{\id_A = \frac{1}{1+\frac{1}{2}\al k}\sum D^\rho}$.
			\qed
		\end{corollary}
		
		\begin{corollary}
			\label{cor ade id}
			The algebra $A^\al_3(G(\rt X_n),D(\rt X_n))$ has an identity
			when ${\rt X_n}$ is of type $\rt A_n,\rt D_n,\rt E_6,\rt E_7,\rt E_8$:
			respectively,
			\begin{gather*}
				\id_{\rt A_n} = \frac{1}{1+\al(n-1)}\sum D(\rt A_n)^\rho, \quad
				\id_{\rt D_n} = \frac{1}{1+\al(2n-4)}\sum D(\rt D_n)^\rho, \\ \quad
				\id_{\rt E_6} = \frac{1}{1+10\al}\sum D(\rt E_6)^\rho, \quad
				\id_{\rt E_7} = \frac{1}{1+16\al}\sum D(\rt E_7)^\rho, \quad
				\id_{\rt E_8} = \frac{1}{1+28\al}\sum D(\rt E_8)^\rho.
				\qed
			\end{gather*}
		\end{corollary}

		As an aside, we remark that idempotents of this kind
		may have application to the (algorithmic) search for idempotents
		in our algebras.
		Since idempotents play a central r\^ole for us
		(and see also the remarks in Section \ref{ss apps}
		relating to Miyamoto's theorem),
		this is an important question.
		It is easy to check that not all idempotents
		are expressible as $\id_B$ or $\id_B - \id_C$
		for $C\subseteq B$ subalgebras of $A$;
		for example, this is visible already in $A^\al_3(\rt A_3)$ \cite{alonso assoc}.
		On the other hand, a large number are,
		and the computational easiness of finding $\id_B$ for arbitrary $B\subseteq A$
		makes this useful.

	\subsection{Eigenvalues and fusion rules for $\rt A_n$}\label{ss eigval}
		We now investigate the behaviour of these identity elements
		embedded in larger algebras,
		in the case of the Weyl group of $\rt A_n$.
		Recall that $G(\rt A_n)$ is $\Sym(n+1)$.
		\begin{lemma}
			\label{lem id eigenvalues}
			Fix an embedding $W(\rt A_{m-1})\subseteq W(\rt A_{n-1})$,
			the latter acting on $\Om = \{1,\dotsc,n\}$
			with $D = D(\rt A_{n-1}) = (1,2)^{W(\rt A_{n-1})}$
			and $E = D(\rt A_{m-1}) = G(\rt A_{m-1})\cap D$.
			Let $\rho = \rho^\al_3\colon D\to A = A^\al_3(\rt A_{n-1})$.
			Then $\id_{E^\rho}$ has
			\begin{align}
				1\text{-eigenvectors:} \quad &
					e^\rho \text{ for } e\in E, \\
				0\text{-eigenvectors:} \quad &
					\label{eq commuting 0 eigvect}
					d^\rho \text{ for } d\in\C_D(E), \\
			\intertext{in $A$, and,
			as $a,b$ range over $\nhood_E(\Om)=\op{Supp}_\Om(\Sym(m))$
			and $z$ over $\C_E(\Om)=\op{Fix}_\Om(\Sym(m))$:}
				0\text{-eigenvectors:} \quad &
				\label{eq id noncomm 0 eigvect}
				\sum_{\mathclap{c\in\nhood_E(\Om)}}(cz)^\rho - \al\id_{E^\rho},\\
				\eta(m)\text{-eigenvectors:} \quad & 
					\label{eq id nontriv eigvect}
					2(1+\al(m-2))\left((az)^\rho-(bz)^\rho\right)
					+ \al(m-2)\sum_{\mathclap{c\in\nhood_E(\Om)\smallsetminus\{a,b\}}}\left((ac)^\rho-(bc)^\rho\right),
			\end{align}
			where $\eta(m)=\smash{\displaystyle\frac{\al m}{2+2\al(m-2)}}$.
			These are all the eigenvectors of $\id_{E^\rho}$ in $A$.
		\end{lemma}
		\begin{proof}
			The $1$-eigenvectors follow by Corollary \ref{cor f3 id}.
			For the $0$-eigenvectors,
			if $d\in D$ commutes with all $e\in E$,
			it follows that $d^\rho e^\rho = 0$
			and hence by linearity $d^\rho\id_{E^\rho}=0$,
			proving \eqref{eq commuting 0 eigvect}.
			For the other $0$-eigenvectors,
			with fixed $(ab)\in E$ and $z\in\C_E(\Om)$ we calculate
			\begin{equation}
				\label{eq erho al eigvect}
				(ab)^\rho\sum_{\mathclap{c\in\nhood_E(\Om)}}(cz)^\rho
				= (ab)^\rho((az)^\rho+(bz)^\rho)
				= \frac{\al}{2}((ab)^\rho+(az)^\rho-(bz)^\rho+(ab)^\rho+(bz)^\rho-(az)^\rho)
				= \al(ab)^\rho,
			\end{equation}
			since $[(ab),(cz)]=1$ unless $c\in\{a,b\}$.
			Extending linearly, it follows that
			\begin{equation}
				\label{eq id mult cz}
				\id_{E^\rho}\Bigl(\sum_{\mathclap{c\in\nhood_E(\Om)}}(cz)^\rho\Bigr)
				= \al\id_{E^\rho}.
			\end{equation}
			By subtracting this term \eqref{eq id mult cz}
			from $\sum_{c\in\nhood_E(\Om)}(cz)^\rho$,
			we obtain the $0$-eigenvector of \eqref{eq id noncomm 0 eigvect}.

			Now fix $a,b\in\nhood_E(\Om)$ as well as $z\in\C_E(\Om)$. Then
			\begin{equation*}
				\begin{aligned}
				\left((az)^\rho - (bz)^\rho\right)\sum E^\rho
				& = \left((az)^\rho - (bz)^\rho\right)\Bigl((ab)^\rho+\sum_{\mathclap{c\in\nhood_E(\Om)\smallsetminus\{a,b\}}}\left((ca)^\rho+(cb)^\rho\right)\Bigr) \\
				& = \left((az)^\rho - (bz)^\rho\right)(ab)^\rho
				+ \sum_{\mathclap{c\in\nhood_E(\Om)\smallsetminus\{a,b\}}}\left((az)^\rho(ca)^\rho-(bz)^\rho(cb)^\rho\right) \\
				& = \frac{\al}{2}\left(2(az)^\rho-2(bz)^\rho\right)
				+ \frac{\al}{2}\sum_{\mathclap{c\in\nhood_E(\Om)\smallsetminus\{a,b\}}}\left((az)^\rho+(ca)^\rho-(cz)^\rho-(bz)^\rho-(cb)^\rho+(cz)^\rho\right) \\
				& = (\al+\frac{\al}{2}(m-2))\left((az)^\rho-(bz)^\rho\right)
				+ \frac{\al}{2}\sum_{\mathclap{c\in\nhood_E(\Om)\smallsetminus\{a,b\}}}\left((ac)^\rho-(bc)^\rho\right).
				\end{aligned}
			\end{equation*}
			Again, we rescale 
			and since $(ac)^\rho-(bc)^\rho$ is a $1$-eigenvector for $\id_{E^\rho}$,
			we can subtract the sum $\sum_{c\in\nhood_E(\Om)\smallsetminus\{a,b\}}((ac)^\rho-(bc)^\rho)$ entirely
			and arrive at the $\eta(m)$-eigenvector in \eqref{eq id nontriv eigvect}.

			It only remains to show that we have found all eigenvectors.
			Denote the $\lm$-eigenspace of $\id_{E^\rho}$ by $A_\lm$
			for the remainder of this proof.
			We see that $E^\rho\leq A_1$ has dimension $\size E = \frac{1}{2}m(m-1)$.
			Also, $\C_D(E)^\rho\leq A_0$ has dimension $\size{\C_D(E)} = \frac{1}{2}(n-m)(n-m-1)$.
			For $z\in\C_E(\Om)$, set $\phi(z)$ to equal \eqref{eq id noncomm 0 eigvect}.
			Then $\phi(\C_E(\Om))\leq A_0$ has dimension $\size{\C_E(\Om)} = n-m$,
			as by inspection different basis vectors occur in each $\phi(z)$.
			Moreover, these are disjoint to the basis vectors supporting $\C_D(E)^\rho$,
			so $\dim A_0 \geq \size{\C_D(E)}+\size{\C_E(\Om)}=\frac{1}{2}(n-m)(n-m-1)+n-m = \frac{1}{2}(n-m)(n-m+1)$.
			Finally, write $\psi(a,b,c)$ to denote the expression \eqref{eq id nontriv eigvect},
			for $a,b\in\nhood_E(\Om)$ and $z\in\C_E(\Om)$.
			Then $\dim\im\psi\geq(m-1)(n-m)$;
			for, fix $a\in\nhood_E(\Om)$,
			and observe that $\dim\psi(a,\nhood_E(\Om),\C_E(\Om))=(m-1)(n-m)$,
			for each expression of the form $\psi(a,b,z)$
			is the unique one in which the basis vector $(bz)^\rho$
			has a nonzero coefficient.
			Now note that the lower bounds for the dimensions of the eigenspaces
			sum to the dimension of the entire space $A$:
			\begin{equation*}
				\frac{1}{2}m(m-1) + \frac{1}{2}(n-m)(n-m+1) + (m-1)(n-m)
				= \frac{1}{2}n(n-1).
				\qedhere
			\end{equation*}
		\end{proof}

		\begin{lemma}
			\label{lem coset eigenvalues}
			The coset axis $x = x_{\Sym(m)/\Sym(\ell)}$ in $A = A^\al_3(\rt A_n)$
			has eigenvalues $1$, $0$ and
			$$ \eta(m)=\frac{\al m}{2+2\al(m-2)},\quad
				1-\eta(\ell)=\frac{2+\al(\ell-4)}{2+2\al(\ell-2)},\quad
				\eta(m)-\eta(\ell) = \frac{2\al(1-2\al)(m-\ell)}{(2+2\al(m-2))(2+2\al(\ell-2))}, $$
			if $n\geq m>\ell\geq 3$;
			it has less eigenvalues for other choices of $n,m,\ell$.
		\end{lemma}
		\begin{proof}
			We first fix some notations:
			as usual, $D = (1,2)^{\Sym(n)}$;
			also denote by $\id_i$ the identity of the subalgebra $\la(D\cap\Sym(i))^\rho\ra$,
			where we understand $\Sym(i)$ to have a fixed embedding in $\Sym(n)$.
			We will show that $\ad(\id_\ell)$ and $\ad(\id_m)$ are simultaneously diagonalisable, satisfying
			$$ A^{\id_\ell}_1 \subseteq A^{\id_m}_1,\quad
				A^{\id_m}_0 \subseteq A^{\id_\ell}_0,\quad
				A^{\id_m}_{\eta(m)} \subseteq A^{\id_\ell}_{0,\eta(\ell)}. $$
			Together with Lemma \ref{lem id eigenvalues},
			this implies the result:
			the possible eigenvectors of $x=\id_m-\id_\ell$ are easily determined,
			and the eigenvalues of $x$ are differences of eigenvalues of $\id_m$ and $\id_\ell$.

			The $1$-eigenspace of $\id_m$
			can be seen in the r\^ole of the algebra $A$ in Lemma \ref{lem id eigenvalues}.
			Thus we have a decomposition of the $1$-eigenspace of $\id_m$ with respect to $\id_\ell$.
			Therefore the eigenvalues of $e = \id_m-\id_\ell$,
			restricted to $A^{\id_m}_1$,
			are precisely $1 - \{1,0,\eta(\ell)\} = \{0,1,1-\eta(\ell)\}$.

			Observe that $\C_{\Sym(\ell)}(D)\supseteq\C_{\Sym(m)}(D)$,
			and hence every $0$-eigenvector of $\id_m$ of type \eqref{eq commuting 0 eigvect}
			is a $0$-eigenvector of $\id_\ell$.
			Furthermore, with $S_i = \op{Supp}_{\Sym(i)}(\{1,\dotsc,n\})$
			and $z\not\in\S_m$,
			\begin{equation}
				\id_\ell\Bigl( \sum_{\mathclap{c\in S_m}}(cz)^\rho-\al\id_m \Bigr)
				= \id_\ell\sum_{\mathclap{c\in S_\ell}}(cz)^\rho + \id_\ell\sum_{\mathclap{c\in S_m-S_\ell}}(cz)^\rho - \al\id_\ell.
			\end{equation}
			Note that, if $c\in S_m-S_z$ and $z\not\in S_m\supseteq S_\ell$
			then $(cz)\in\C_D(\Sym(\ell))$ and $\id_\ell (cz)^\rho = 0$,
			so the middle term may be cancelled.
			The two remaining terms are recognised as equal
			by \eqref{eq id mult cz}, and hence cancel.
			So $0$-eigenvectors of type $\eqref{eq id noncomm 0 eigvect}$ for $\id_m$
			are also $0$-eigenvectors for $\id_\ell$.

			We may finally treat $\eta(m)$-eigenvectors of $\id_m$.
			Rather than go through an explicit computation,
			observe that $A = A^{\id_m}_1 + A^{\id_m}_0 + A^{\id_m}_{\eta(m)}$,
			and since we have shown that $A^{\id_m}_1$ and $A^{\id_m}_0$
			are $\ad(\id_\ell)$-modules and $A$ is a vector space, hence splits,
			we have that $A^{\id_m}_{\eta(m)}$ is an $\ad(\id_\ell)$-module.
			It follows by Lemma \ref{lem id eigenvalues}
			that it decomposes into $1$-, $0$- and $\eta(\ell)$-eigenspaces.
			Since the $1$-eigenspace of $\id_\ell$ is a subspace of $A^{\id_m}_1$,
			we have that $A^{\id_m}_{\eta(m)}$ decomposes solely into $0$- and $\eta(\ell)$-eigenspaces.
			This gives the eigenvalues $\eta(m)-0=\eta(m)$
			and $\eta(m)-\eta(\ell)$.
		\end{proof}

		\begin{corollary}
			\label{cor ca primitive}
			The coset axis $x = x_{\Sym(m)/\Sym(\ell)}$ is primitive
			only if $m = \ell+1$.
		\end{corollary}
		\begin{proof}
			We continue to use the notation of the previous Lemma.
			The $1$-eigenspace $A^x_1$ of $x$ in $A$
			contains $A^{\id_m}_1\cap A^{\id_\ell}_0$.
			By Lemma \ref{lem id eigenvalues},
			this has dimension $\size{S_m-S_\ell} + \size{\C_{D\cap\Sym(m)}(D\cap\Sym(\ell))}$;
			evidently this is $1$ if and only if $m=\ell+1$.
		\end{proof}

		\clearpage
		\begin{theorem}
			\label{thm fus rules}
			The coset axis $x_{\Sym(m)/\Sym(\ell)}$ satisfies the fusion rules $\Phi$:
			\begin{center}
				\begin{tabular}{|c||c|c|c|c|c|}
					\hline
						$\star$ & $1$ & $0$ & $\eta(m)$ & $1-\eta(\ell)$ & $\eta(m)-\eta(\ell)$\\
					\hline\hline
						$1$ & $\{1\}$ & $\emptyset$ & $\{\eta(m)\}$ & $\{1-\eta(\ell)\}$ & $\{\eta(m)-\eta(\ell)\}$ \\
					\hline
						$0$ &  & $\{0\}$ & $\{\eta(m)\}$ & $\{1-\eta(\ell)\}$ & $\{\eta(m)-\eta(\ell)\}$ \\
					\hline
						$\eta(m)$ &  &  & $\{1,0,\eta(m)\}$ & $\{\eta(m)-\eta(\ell)\}$ & $\{1-\eta(\ell),\eta(m)-\eta(\ell)\}$ \\
					\hline
						$1-\eta(\ell)$ &  &  &  & $\{1,0,1-\eta(\ell)\}$ & $\{\eta(m),\eta(m)-\eta(\ell)\}$ \\
					\hline
						$\eta(m)-\eta(\ell)$ & 	&  &  &  & $\{1,0,\eta(m),1-\eta(\ell),\eta(m)-\eta(\ell)\}$ \\
					\hline
				\end{tabular}
			\end{center}
		\end{theorem}
		\begin{proof}
			We first calculate the fusion rules $(\Phi',\bullet)$
			for eigenvectors of $\id_m = \id_{E^\rho}$,
			for $(\Sym(m),E)\subseteq(\Sym(n),D)$, $D=(1,2)^{\Sym(n)}$,
			in $A^\al_3(\Sym(n),D)$.
			These turn out to be
			\begin{center}
				\begin{tabular}{|c||c|c|c|}
					\hline
						$\bullet$	& $1$ & $0$ & $\eta(m)$ \\
					\hline\hline
						$1$ & $\{1\}$ & $\emptyset$ & $\{\eta(m)\}$ \\
					\hline
						$0$ &  & $\{0\}$ & $\{\eta(m)\}$ \\
					\hline
						$\eta(m)$ &  &  & $\{1,0,\eta(m)\}$ \\
					\hline
				\end{tabular}
			\end{center}
			From this, the fusion rules $\Phi$ will follow:
			every eigenvector of $x$
			is an eigenvector for $\id_m$ and an eigenvector for $\id_\ell$,
			by Lemma \ref{lem coset eigenvalues}.
			Hence, for two eigenvectors $u$ and $v$ of $x$
			with associated eigenvalues $\kp,\lm$ and $\mu,\nu$
			for $\ad(\id_m),\ad(\id_\ell)$ respectively,
			we have that
			\begin{equation} uv\in
				(A^{\id_m}_\kp A^{\id_m}_\mu)\cap (A^{\id_\ell}_\lm A^{\id_\ell}_\nu)
				= A^{\id_m}_{\kp\bullet\mu}\cap A^{\id_\ell}_{\lm\bullet\nu}
				= A^x_{\kp\bullet\mu-\lm\bullet\nu},
			\end{equation}
			with further restrictions arising from the containment of eigenspaces
			described in Lemma \ref{lem coset eigenvalues}.
			Therefore $(\kp-\lm)\star(\mu-\nu)=\kp\bullet\mu-\lm\bullet\nu$,
			where $S-T$ is the pointwise difference $\{s-t\mid s\in S,t\in T\}$ of sets $S,T$.
			We exhibit two exemplary deductions:
			\begin{gather}
				\begin{aligned}
					0\star\eta(m)
					& = (1\bullet\eta(m)-1\bullet0)\cup(0\bullet\eta(m)-0\bullet0) \\
					& = (\{\eta(m)\}-\emptyset)\cup(\{\eta(m)\}-\{0\})
					= \emptyset\cup\{\eta(m)\}
					= \{\eta(m)\},
				\end{aligned} \\
				\begin{aligned}
					\eta(m)\star(\eta(m)-\eta(\ell))
					& = \eta(m)\bullet\eta(m)-0\bullet\eta(l)
					= \{1,0,\eta(m)\}-\{\eta(\ell)\} \\
					& = \{1-\eta(\ell),\eta(m)-\eta(\ell)\}.
				\end{aligned}
			\end{gather}

			Now to prove the rules we stated for $\bullet$.
			Note that there is nothing to prove to show that $\eta(m)\bullet\eta(m)=\{1,0,\eta(m)\}$.
			It is clear that $E^\rho E^\rho\subseteq E^\rho$,
			so $1\bullet1=\{1\}$.
			Set $S = \op{Supp}(\Sym(m))$.
			Also $1\bullet0=\emptyset$,
			for if $e\in E$ and $d\in\C_D(E)$ we have $e^\rho d^\rho=0$,
			and by \eqref{eq erho al eigvect},
			\begin{equation}
				e\Bigl(\sum_{\mathclap{c\in S}}(cz)^\rho-\al\id_{E^\rho}\Bigr)
				= \al e-\al e = 0.
			\end{equation}

			We will observe that $0\bullet0=\{0\}$ case by case.
			As before, $\C_D(E)^\rho\C_D(E)^\rho\subseteq\C_D(E)$,
			which demonstrates that the product of two $0$-eigenvectors
			of types \eqref{eq commuting 0 eigvect} is again a $0$-eigenvector.
			If $(yy')\in\C_D(E)$ is arbitrary then
			\begin{align}
				(yy')^\rho&\sum_{\mathclap{c\in S}}(cz)^\rho - (yy')^\rho\al\id_{E^\rho}
				\intertext{is $0$ when $y'\neq z$, because all elements commute;
					if however $y' = z$, then}
				&= \frac{\al}{2}\sum_{\mathclap{c\in S}}((cz)^\rho+(yz)^\rho-(cy)^\rho)
				= \frac{\al}{2}\size{ S}(yz) + \frac{\al}{2}\sum_{\mathclap{c\in S}}((cz)^\rho-(cy)^\rho),
			\end{align}
			and of course $(yz)$ is a $0$-eigenvector for $\id_{E^\rho}$,
			and using \eqref{eq id mult cz},
			\begin{equation}
				\id_{E^\rho}\sum_{\mathclap{c\in S}}((cz)^\rho-(cy)^\rho)
				= \id_{E^\rho}\sum_{\mathclap{c\in S}}(cz)^\rho
					- \id_{E^\rho}\sum_{\mathclap{c\in S}}(cy)^\rho
				= \al\id_{E^\rho} - \al\id_{E^\rho} = 0.
			\end{equation}
			Therefore the product of $0$-eigenvectors of types \eqref{eq commuting 0 eigvect} and \eqref{eq id noncomm 0 eigvect} is also a $0$-eigenvector.
			Finally, for $y,z\not\in S$,
			we compute the product of two $0$-eigenvectors of type \eqref{eq id noncomm 0 eigvect}:
			\begin{align}
				\label{eq prod type II 0eigvect}
				\Bigl(\sum_{a\in S} & (ay)^\rho-\al\id_{E^\rho}\Bigr)
				\Bigl(\sum_{b\in S}(az)^\rho-\al\id_{E^\rho}\Bigr)
				= \Bigl(\sum_{\substack{a\in S\\b\in S}}(ay)^\rho(bz)^\rho\Bigr)
				- \al^2\id_{E^\rho}
				\intertext{Supposing that $y\neq z$, that equals}
				& = \Bigl(\sum_{a\in S}(ay)^\rho(az)^\rho\Bigr) - \al^2\id_{E^\rho}
				= \Bigl(\frac{\al}{2}\sum_{a\in S}(ay)^\rho+(az)^\rho-(yz)^\rho\Bigr) - \al^2\id_{E^\rho}
				\intertext{and, when we multiply by $\id_{E^\rho}$,
				this kills the terms $(yz)^\rho$ and leaves us with
				$\frac{\al}{2}(\al\id_{E^\rho}+\al\id_{E^\rho}) - \al^2\id_{E^\rho}=0$.
				On the other hand, suppose $y=z$ and continue from \eqref{eq prod type II 0eigvect}: }
				& = \Bigl(\sum_{a\in S}(az)^\rho + \frac{\al}{2}\sum_{b\in S}((az)^\rho+(bz)^\rho-(ab)^\rho)\Bigr) - \al^2\id \\
				& = \sum_{a\in S}(az)^\rho - \al\sum E^\rho + \al(m-1)\sum_{a\in S}(az)^\rho - \al\id_{E^\rho}.
			\end{align}
			Multiplying by $\id_{E^\rho}$, we get
			$\al\id_{E^\rho} - \al\sum E^\rho +\al^2(m-1)\id_{E^\rho} - \al^2\id_{E^\rho}
			= \al\left(\sum E^\rho-\sum E^\rho\right) = 0$.

			The remaining cases $1\bullet\eta(m),0\bullet\eta(m)$ we omit.
			They are along the same lines as the previous calculations.
		\end{proof}

	\subsection{Central Charges} \label{ss cc}
		The {\em central charge} $\op{cc}(x)$ of an idempotent $x$
		in an algebra with bilinear form $\la,\ra$
		is defined to be $\frac{1}{2}\la x,x\ra$.
		(The name, and the factor $\frac{1}{2}$,
		comes from extension of the Virasoro-VOA case,
		as we will explore below.)
		Also, a unital algebra $A$
		has central charge $\op{cc}(A) = \op{cc}(\id_A) = \frac{1}{2}\la\id_A,\id_A\ra$.

		In \eqref{eq def form}, we remarked that
		an algebra $A = A^\al_3(G,D)$ admits a bilinear form.
		We use this to study coset axes.

		\begin{lemma}
			\label{lem cc id}
			Suppose that $A = A^\al_3(G,D)$ for a $k$-regular $(G,D)$.
			Then $\op{cc}(A)=\displaystyle\smash{\frac{\size D}{2+\al k}}$.
		\end{lemma}
		\begin{proof}
			By straightforward calculation:
			\begin{align*}
				\op{cc}(A) & = \frac{1}{2}\la\id_A,\id_A\ra
				= \frac{1}{2}\frac{1}{(1+\frac{1}{2}\al k)^2}\Bigl\la\sum D^\rho,\sum D^\rho\Bigr\ra \\
				& = \frac{1}{2(1+\frac{1}{2}\al k)^2}\sum_{d\in D}\Bigl(\la d^\rho,d^\rho\ra + \sum_{e\in\nhood_D(d)}\la d^\rho,e^\rho\ra + \sum_{e\in\commute_D(d)}\la d^\rho,e^\rho\ra\Bigr) \\
				& = \frac{1}{2(1+\frac{1}{2}\al k)^2}\sum_{d\in D}(1+k\frac{\al}{2}+0)
				= \frac{\size D}{2+\al k}.
				\qedhere
			\end{align*}
		\end{proof}

		\begin{corollary}
			\label{cor cc weyl}
			If ${\rt X_n}$ is a simply-laced root system, then
			$\displaystyle\smash{\op{cc}(A^\al_3({\rt X_n}))=\frac{n h_{\rt X_n}}{4(1+\al(h_{\rt X_n}-2))}}$.
			\qed
		\end{corollary}

		\begin{lemma}
			\label{lem assoc ids}
			Suppose that $(K,F)\leq(H,E)\leq(G,D)$ is a chain of transposition groups,
			with $(K,F)$ and $(H,E)$ $k_F$ and $k_E$-regular respectively.
			Then, in $A^\al_3(G,D)$,
			$$ \la \id_{E^\rho},\id_{F^\rho}\ra = \la \id_{F^\rho},\id_{F^\rho}\ra. $$
		\end{lemma}
		\begin{proof}
			The noncommuting elements outside of $F$
			make up the difference in scaling between $\id_E$ and $\id_F$.
			We calculate
			\begin{align*}
				\la \id_{E^\rho},\id_{F^\rho}\ra
				& = \frac{1}{1+\frac{1}{2}\al k_E}\frac{1}{1+\frac{1}{2}\al k_F}\Bigl\la\sum E^\rho,\sum F^\rho\Bigr\ra \\
				& = \frac{1}{(1+\frac{1}{2}\al k_E)(1+\frac{1}{2}\al k_F)}\sum_{f\in F}\Bigl(\sum_{\mathclap{f'\in F}}\la f,f'\ra+\sum_{\mathclap{e\in\nhood_E(F)}}\la f,e\ra + \sum_{\mathclap{e\in\commute_E(F)}}\la f,e\ra\Bigr) \\
				& = \frac{1}{(1+\frac{1}{2}\al k_E)(1+\frac{1}{2}\al k_F)}\sum_{f\in F}\left(1+\frac{1}{2}\al k_F + (k_E-k_F)\frac{\al}{2}+0\right) \\
				& = \frac{\size F(1+\frac{1}{2}\al k_E)}{(1+\frac{1}{2}\al k_E)(1+\frac{1}{2}\al k_F)}
				= \frac{\size F}{1+\frac{1}{2}\al k_F}.
				\qedhere
			\end{align*}
		\end{proof}

		\begin{corollary}
			\label{cor cc ca}
			The central charge of the coset axis $x_{H/K}$ is
			$\displaystyle\smash{\frac{\size E}{2+\al k_E}-\frac{\size F}{2+\al k_F}}$.
			\qed
		\end{corollary}

		\begin{corollary}
			\label{cor cc sym cosets}
			The coset axis $x_m = x_{\rt A_m/\rt A_{m-1}}$ has central charge
			\begin{equation}
				\op{cc}(x_m) = \frac{m(2+\al(m-3))}{4(1+\al(m-1))(1+\al(m-2))};
			\end{equation}
			the coset axis $y_m = y_{\rt D_m/\rt D_{m-1}}$, for $m>4$,
			has central charge
			\begin{equation}
				\op{cc}(y_m) = \frac{(m-1)(1+\al(m-4))}{(1+\al(2m-4))(1+\al(2m-6))}. 
			\end{equation}
		\end{corollary}
		\begin{proof}
			We can apply Corollary \ref{cor cc ca}
			using that $W(\rt X_n)$ is $2h_{\rt X_n}-4$-regular.
		\end{proof}

		Now, for all $\al\in\RR$, we define
		\begin{align}
			f^\rt A_\al\colon\RR\to\RR,\quad
			f^\rt A_\al(m) & = \frac{m(2+\al(m-3))}{4(1+\al(m-1))(1+\al(m-2))}, \\
			f^\rt D_\al\colon\RR\to\RR,\quad
			f^\rt D_\al(m) & = \frac{(m-1)(1+\al(m-4))}{(1+\al(2m-4))(1+\al(2m-6))}.
		\end{align}
		Note that $f^\rt A_\al(1)=\frac{1}{2} = f^\rt D_\al(1)$,
		that is, for all values of $\al$;
		this corresponds to $\id_{\Sym(1)}$ being an axis of central charge $\frac{1}{2}$ in our setup
		and allowing $\id_{\Sym(0)}=0$.
		Also observe
		\begin{equation}
			\lim_{m\to\infty} f^\rt A_\al(m)
				= \frac{1}{4\al}
				= \lim_{m\to\infty} f^\rt D_\al(m).
		\end{equation}

		It is not clear, at this stage, whether or not
		the form $\la,\ra$ on $A^\al_3(G,D)$ is positive-definite.
		But supposing we desire positive-definiteness,
		certainly we want $f^\rt A_\al(\ZZ),f^\rt D_\al(\ZZ> 4)$
		to be bounded below by $0$.
		Clearly then $\al\geq0$.

		The limit of $f^\rt A_\al(m)$ suggests a possible specialisation of $\al$:
		if we desire $\op{cc}(x_m)$ to be bounded by $1$ for all $m$,
		the smallest possible $\al$ is $\al = \frac{1}{4}$.
		In this case we have
		\begin{align}
			\label{eq f quarter m}
			f^\rt A_{1/4}(m) & = \frac{m(m+5)}{(m+2)(m+3)} = 1 - \frac{6}{(m+2)(m+3)} \\
			\label{eq f quarter Dm}
			f^\rt D_{1/4}(m) & = \frac{m(m-1)}{m(m-1)} = 1.
		\end{align}
		We also calculate
		\begin{equation}
			\label{eq cc 1/32}
			f^\rt A_{1/32}(m) = \frac{8m(m+61)}{(m+30)(m+31)}
			= 8\left(1-\frac{930}{(m+30)(m+31)}\right).
		\end{equation}

	\subsection{Relations to Virasoro representations} \label{ss apps}
		For certain choices of $(G,D)$ and $\al$,
		it turns out that the algebra $A = A^\al_3(G,D)$
		lies inside a VOA
		and carries representations of the Virasoro algebra.
		
		In particular,
		let ${\rt X_n}$ be a simply-laced root system.
		Then \cite{flm}, in points 8.9.5 and 8.9.7,
		provides a construction of a so-called
		lattice vertex operator algebra $V({\rt X_n})$.
		In particular, its weight-$2$ subspace $V({\rt X_n})_2$
		is closed under the product $\cdot_1\cdot$
		and admits an associating bilinear form by $\cdot_3\cdot$;
		see also \cite{dlmn}.
		We present a generalised version $\hat A = \hat A^\al_3(\rt X_n)$ of $V(\rt X_n)_2$,
		which contains a copy of $A^\al_3(\rt X_n)$.
		In effect, if the $\Phi^\al_3$-axes
		in $A^\al_3(\rt X_n)$ correspond to a set of positive roots of $\rt X_n$,
		then $\hat A^\al_3(\rt X_n)$ contains also idempotents corresponding to negative roots.

		A basis of $\hat A = \hat A^\al_3(G,D)$ may be parametrised
		as $\{ d^\rho  =  d^\rho _+, d^\rho _-\mid d\in D \}$,
		with algebra product and form, for $c,d\in D$ and $\ep,{\ep'}\in\{+,-\}$,
		\begin{equation}
			\label{eq double fischer}
			c^\rho_\ep d^\rho_{\ep'} = \begin{cases}
				d^\rho  & \text{ if } c = d \text{ and } \ep = {\ep'}, \\
				\frac{\al}{2}( c^\rho _\ep +  d^\rho _{\ep'} - (c^d)^\rho_{\ep{\ep'}})
				 & \text{ if } [c,d]\neq 1, \\
				0 & \text{ otherwise.}
			\end{cases}
			\quad
			\la c^\rho_\ep, d^\rho_{\ep'}\ra  = \begin{cases}
				1 & \text{ if } c = d \text{ and } \ep = {\ep'}, \\
				\frac{\al}{2} & \text{ if } [c,d]\neq 1, \\
				0 & \text{ otherwise.}
			\end{cases}
		\end{equation}
		Theorem 3.1 of \cite{dlmn} is that there exists an isometric surjection
		from $\hat A^{1/4}_3(\rt X_n)$ onto $V(\rt X_n)_2$;
		furthermore, the surjection is given by quotienting the radical of the form,
		which is trivial precisely when $\rt X_n = \rt A_n$.
		
		Recall that Miyamoto's theorem \cite{miyamoto}
		states that any idempotent $e$ in a moonshine-type VOA $V$
		carries a representation of the Virasoro algebra
		of central charge equal to the central charge of the idempotent.
		This means that the subalgebra $A = V_2$
		admits a decomposition $A = \la e\ra \oplus \bigoplus_{h \text{ h.w. }\op{Vir}_c}A^{2e}_h$,
		where $h$ ranges over the highest weights of $\op{Vir}_c$
		and $A^{2e}_h$ is, as usual, the (possibly empty) $h$-eigenspace of $\ad({2e})$ on $A$.
		Furthermore, the eigenspace $A^{2e}_h$ is an isotypic component of the representation,
		each $1$-dimensional subspace forming part of an irreducible representation of highest weight $h$.
		The fusion rules,
		similar to Clebsch--Gordan coefficients for tensor products of irreducible modules,
		dictate the fusion of irreducible modules under the constraint of fixed central charge,
		and hence also restrict in which eigenspaces the product of two eigenvectors may lie.
		This is of particular interest when $\op{Vir}_c$ has finitely many irreducible modules,
		\ie highest weights,
		namely, for $c = c(p,q)$ in \eqref{eq cc(p,q) vir},
		lying between $0$ and $1$.

		From Miyamoto's theorem we deduce that
		the coset axes (and identities) of the algebra $A^{1/4}_3(\rt X_n)$
		carry representations of the Virasoro algebra.
		The results of \cite{flm} showing the existence of the VOA $V(\rt X_n)$
		and its subalgebra $A = A^{1/4}_3(\rt X_n)$,
		together with our analysis of $A$,
		give a new formula for some highest weights of $\op{Vir}_c$
		for varying values of $c = c(m+1,m)$,
		and indicates a subset of the fusion rules
		for the relevant irreducible representations.

		In particular, from \eqref{eq cc(p,q) vir}
		and Corollary \ref{cor cc sym cosets} via \eqref{eq f quarter m}
		we deduce that $x_m=x_{A_m/A_{m-1}}$
		carries a representation of $\op{Vir}_{c(m+3,m+2)}$.
		We also know that its eigenvalues are $1,0,\eta(m+1),1-\eta(m)$ and $\eta(m+1)-\eta(m)$.
		Specialised to the case $\al=\frac{1}{4}$ from Lemma \ref{lem coset eigenvalues},
		we have $\eta_{1/4}(m) =\displaystyle{\frac{m}{2(m+2)}}$,
		and the eigenvalues of $x_m$ are $1, 0$ and
		\begin{equation}
			\eta_{1/4}(m+1) = \frac{m+1}{2(m+3)},\quad
			1-\eta_{1/4}(m) = \frac{m+4}{2(m+2)},\quad
			\eta_{1/4}(m+1)-\eta_{1/4}(m) = \frac{1}{(m+2)(m+3)}.
		\end{equation}
		\begin{observation}
			\label{obs hw identification}
			By inspection (confer \eqref{eq hw vir}), we find that
			\begin{equation}
				\begin{aligned}
					\label{eq hw identification}
					0 = \frac{1}{2}h^{c(m+3,m+2)}_{1,1},\\
					\eta_{1/4}(m+1) = \frac{1}{2}h^{c(m+3,m+2)}_{3,1},\\
					1-\eta_{1/4}(m) = \frac{1}{2}h^{c(m+3,m+2)}_{1,3},\\
					\eta_{1/4}(m+1)-\eta_{1/4}(m) = \frac{1}{2}h^{c(m+3,m+2)}_{3,3}.
				\end{aligned}
			\end{equation}
		\end{observation}
		(Recall that the factor of $\frac{1}{2}$ is an artifact of different choices of scaling.
		For related reasons, eigenvalue $1$ does not enter the picture.)

	\subsection{Additional eigenvalues in $\hat A^\al_3(\rt A_n)$} \label{ss double fischer}
		We will observe that the coset axis $x_{\rt A_m/\rt A_{m-1}}$
		has up to seven distinct eigenvalues in $\hat A^\al_3(\rt A_n)$.
		By extending the work of Section \ref{ss eigval},
		we find the two additional eigenvalues in Corollary \ref{cor hat eigv},
		and make deductions along the lines of Section \ref{ss apps}
		culminating in Obsevation \ref{obs hat hw identification}.

		\begin{lemma}
			\label{lem id+ d-}
			In $A = A^\al_3(G,D) \subseteq \hat A = \hat A^\al_3(G,D)$,
			for a $k$-regular $(G,D)$,
			$$ d^\rho_- \id_A = \frac{\al}{2+\al k}\Bigl(kd^\rho_- + \sum_{\mathclap{c\in\nhood_D(d)}}(c^\rho_+ - c^\rho_-)\Bigr). $$
		\end{lemma}
		\begin{proof}
			Using \eqref{eq nhood 3trgp} and \eqref{eq double fischer},
			\begin{equation}
				\begin{aligned}
					d^\rho_-\sum D^\rho_+
					 = d^\rho_-\Bigl(d^\rho_+ + \sum_{\mathclap{c\in\nhood_D(d)}}c^\rho_+ + \sum_{\mathclap{c\in\commute_D(d)}}c^\rho_+\Bigr)
					 = \frac{\al}{2}\sum_{\mathclap{c\in\nhood_D(d)}}(d^\rho_- + c^\rho_+ - (c^d)^\rho_-)
					 = \frac{\al}{2}\size{\nhood_D(d)}d^\rho_- + \frac{\al}{2}\sum_{\mathclap{c\in\nhood_D(d)}}(c^\rho_+ - (d^c)^\rho_-).
				\end{aligned}
			\end{equation}
			Note that $\{d^c\mid c\in\nhood_D(d)\} = \nhood_D(d)$,
			so summing over $d^c$ for all $c\in\nhood_D(d)$
			is the same as summing over $c\in\nhood_D(d)$.
		\end{proof}

		\begin{lemma}
			\label{lem id+ extra eigv}
			Suppose that $a,b,c,d\in D = D(\rt A_{m-1})$
			such that $\la a,b\ra,\la c,d\ra\cong {C_2}^2$ have equal support.
			Then
			\begin{equation}
				\label{eq hateta vect}
				\al(a^\rho_++b^\rho_+-c^\rho_+-d^\rho_+)
				+(1-\al)(a^\rho_-+b^\rho_--c^\rho_--d^\rho_-)
			\end{equation}
			is a $\hat\eta(m) =\displaystyle\frac{\al(m-1)}{1+\al(m-2)}$-eigenvector of $\id_A$.
		\end{lemma}
		\begin{proof}
			From Lemma \ref{lem id+ d-} we get an expansion
			\begin{equation*}
				\begin{aligned}
					\id_A & (a_-+b_--c_--d_-) \\
						& = \frac{\al}{2+\al k}\left(k(a_-+b_--c_--d_-)
							+ \sum_{\mathclap{x\in\nhood_D(a)}}(x_+-x_-)
							+ \sum_{\mathclap{x\in\nhood_D(b)}}(x_+-x_-)
							- \sum_{\mathclap{x\in\nhood_D(c)}}(x_+-x_-)
							- \sum_{\mathclap{x\in\nhood_D(d)}}(x_+-x_-)\right) \\
						& = \frac{\al}{2+\al k}\left(
							(k+2)(a_-+b_--c_--d_-)
							- 2(a_++b_+-c_+-d_+) \right).
				\end{aligned}
			\end{equation*}
			The second equality follows as follows.
			Suppose $x\in\nhood_D(a)\cup\nhood_D(b)\cup\nhood_D(c)\cup\nhood_D(d)$.
			If $x\not\in\{a,b,c,d\}$ then $x$ does not commute with precisely two of $\{a,b,c,d\}$,
			and these two themselves do not commute;
			the contributions of $x$ then cancel.
			Otherwise, $x\in\{a,b,c,d\}$
			and does not commute with two commuting elements,
			and $2(x^\rho_+-x^\rho_-)$ remains.
			Evidently $a^\rho_++b^\rho_+-c^\rho_+-d^\rho_+$ is a $1$-eigenvector for $\id_A$,
			so we can solve to find the result.
		\end{proof}

		\begin{lemma}
			\label{lem hat eta containment}
			The $\hat\eta(m+1)$-eigenspace of $\id_{\rt A_m}$
			decomposes into $\eta(m)$- and $\hat\eta(m)$-eigenvectors for $\id_{\rt A_{m-1}}$.
		\end{lemma}
		\begin{proof}
			Take $D = D(\rt A_m)$ and $E = D(\rt A_{m-1})$,
			$a,b,c,d\in D$ as in Lemma \ref{lem id+ extra eigv}
			composing the $\hat\eta(m+1)$-eigenvector $x$ of $\id_m$,
			and let $\Om$ be the support of $a,b,c,d$.
			Then either $\Om$ is contained in the support of $E$,
			in which case \eqref{eq hateta vect} is a $\hat\eta(m)$-eigenvector
			for $\id_m = \id_{\rt A_m}$.
			Or $\size{\Om\cap\op{Supp}(E)}=3$,
			and we suppose without loss of generality that $a,c\in E$
			and $b,d\not\in E$.
			Then there exist $b',d'\in E$ such that
			$a,b',c,d'$ are as in Lemma \ref{lem id+ extra eigv},
			so
			\begin{equation}
				x' = \al(a^\rho_+ + b'^\rho_+ - c^\rho_+ - d'^\rho_+)
				+ (1-\al)(a^\rho_- + b'^\rho_- - c^\rho_- - d'^\rho_-)
			\end{equation}
			is a $\hat\eta(m)$-eigenvector for $\id_m$.
			Also
			\begin{equation}
				x'' = \al(a^\rho_+ + 2b^\rho_+ - b'^\rho_+ -c^\rho_+ -2d^\rho_+ + d'^\rho_+)
				- (1-\al)(a^\rho_- + 2b^\rho_- - b'^\rho_- -c^\rho_- -2d^\rho_- - d'^\rho_-)
			\end{equation}
			is a $\eta(m)$-eigenvector for $\id_m$,
			and it is clear that $x = \frac{1}{2}(x' + x'')$.
		\end{proof}

		Note that the last lemma would fail
		for $\id_m$ and $\id_\ell$ if $\ell < m - 1$.
		In that case, many more eigenvalues are possible and do occur.

		\begin{corollary}
			\label{cor hat eigv}
			The eigenvectors of $x_m = x_{\rt A_m/\rt A_{m-1}}$ in $\hat A^\al_3(\rt A_m)$ are
			\begin{equation*}
				1,\quad
				0,\quad
				\eta(m+1),1-\eta(m),\quad
				\eta(m+1)-\eta(m),\quad
				\hat\eta(m+1)-\eta(m),\quad
				\hat\eta(m+1)-\hat\eta(m).
				\qed
			\end{equation*}
		\end{corollary}

		We specialise the new eigenvectors to the case $\al = \frac{1}{4}$:
		as $\displaystyle\smash{\hat\eta_{1/4}(m) = \frac{m-1}{m+2}}$ in this case,
		\begin{equation}
			\label{eq hat 1/4 eigvals}
			\begin{aligned}
				\hat\eta_{1/4}(m+1) - \eta_{1/4}(m) & = \frac{m(m+1)}{2(m+2)(m+3)}, \\
				\hat\eta_{1/4}(m+1) - \hat\eta_{1/4}(m) & = \frac{3}{(m+2)(m+3)}.
			\end{aligned}
		\end{equation}
		It is now a simple matter for us to complement Observation \ref{obs hw identification} with
		\begin{observation}
			\label{obs hat hw identification}
			We see that
			\begin{equation}
				\begin{aligned}
					\label{eq hat hw identification}
					\hat\eta_{1/4}(m+1)-\eta_{1/4}(m) & = \frac{1}{2}h^{c(m+3,m+2)}_{5,3}, \\
					\hat\eta_{1/4}(m+1)-\hat\eta_{1/4}(m) & = \frac{1}{2}h^{c(m+3,m+2)}_{5,5}. \\
				\end{aligned}
			\end{equation}
		\end{observation}

\end{document}